\documentclass[12pt]{article}

\usepackage{amsmath, amsthm, amssymb,hyperref,a4wide}
\usepackage{authblk}

\newcommand{\N}{\mathbb{N}}
\newcommand{\Z}{\mathbb{Z}}
\newcommand{\R}{\mathbb{R}}
\newcommand{\C}{\mathbb{C}}

\newcommand{\Hs}{\mathcal{H}}

\newcommand{\Sh}{\mathcal{S}}
\newcommand{\inner}[2]{\langle #1, #2 \rangle}
\newcommand{\plig}{\textbf{plig }}
\newcommand{\norm}[1] {\| #1 \|}
\newcommand{\supp}[1]{\mathrm{Supp}(#1)}
\newcommand{\set}[1]{\left\{ #1 \right\}}

 \newtheorem{theorem}{Theorem}[section]
 \newtheorem{defn}[theorem]{Definition}
 \newtheorem{prop}[theorem]{Proposition}
\newtheorem{lemma}[theorem]{Lemma}
\newtheorem{corollary}[theorem]{Corollary}

\newtheorem{remark}[theorem]{Remark}

 \newtheorem{convention}[theorem]{Convention}

\begin{document}
\title{Embeddings of locally compact hyperbolic groups into $L_p$-spaces}
\author{Chris Cave \thanks{The first author is sponsored by the EPSRC, grant number EP/I0
16945/1. \url{cc1g11@soton.ac.uk} } \ and Dennis Dreesen \thanks{The second author is supported by the Marie Curie 7th framework programme (FP7). The author is a Marie Curie IEF research fellow. \url{dennis.dreesen@soton.ac.uk}.}}
\affil{University of Southampton, Building 54: Mathematics, University Road, SO17 1BJ Southampton, United Kingdom}

\maketitle
\abstract{In the last years, there has been a large amount of research on embeddability properties of finitely generated hyperbolic groups. In this paper, we elaborate on the more general class of locally compact hyperbolic groups. We compute the equivariant $L_p$-compression in a number of locally compact examples, such as the groups $SO(n,1)$.
Next, we show that although there are locally compact, non-discrete hyperbolic groups $G$ with Kazhdan's property ($T$), it is true that any locally compact hyperbolic group admits a proper affine isometric action on an $L_p$-space for $p$ larger than the Ahlfors regular conformal dimension of $\partial G$. This answers a question asked by Yves de Cornulier. Finally, we elaborate on the locally compact version of property $(A)$ and show that, as in the discrete case, it is connected to the non-equivariant Hilbert space compression of a group.\\

}

\section{Introduction}
\subsection{Locally compact hyperbolic groups}
The common convention when dealing with hyperbolic groups is that such groups are finitely generated and equipped with the word length metric relative to a finite generating subset. On the one hand, this leads to a very interesting theory with very strong results, but on the other hand, this class misses interesting groups such as $SO(n,1), SU(n,1), Sp(n,1)$ that contain hyperbolic uniform lattices. Gromov's work \cite{Gromov:1} already contained ideas which encompass locally compact hyperbolic groups. Following \cite{CT} we define a locally compact group $G$ to be {\bf hyperbolic} if it is compactly generated and word-hyperbolic with respect to the word length metric relative to some compact generating subset. This is equivalent to the group acting continuously, properly, cocompactly and isometrically on a proper hyperbolic geodesic metric space $X$ (see Corollary 2.6 of \cite{Caprace}). The space $X$ is determined up to quasi-isometry and so one can unambiguously define the hyperbolic boundary $\partial G$ of $G$ as the hyperbolic boundary of $X$. 
Some examples of locally compact hyperbolic groups are $SO(n,1),SU(n,1),Sp(n,1),F_4^{-20}$ and the class of groups of the form $H\rtimes_\alpha \Z, \ H\rtimes_\alpha \R$ where $\alpha(1)$ satisfies a suitable contracting property \cite{Caprace}.

There are many, also non-trivial, results that generalize to the locally compact setting. For example, in \cite{Carette-Dreesen}, the authors generalize Bowditch's topological characterization of discrete hyperbolic groups to the locally compact setting. This for starters has a nice application in the study of sharply-$n$-transitive actions on compact sets. All groups acting sharply-$n$-transitively on compact spaces $M$ are completely classifed, except for $n=3$. The authors show that $\sigma$-compact groups acting sharply-3-transitively on a compact space $M$ are necessarily locally compact hyperbolic, that $M$ is homeomorphic to the boundary of the group and that the action on $M$ coincides, via this homeomorphism, with the natural action of a hyperbolic group on its boundary.

On the other hand, as the discrete hyperbolic groups merely constitute a special case of the more general locally compact setting, one can not simply hope to extend all results from the discrete to the locally compact context. 
For example, from the discrete case, the intuition has grown that hyperbolicity and amenability are somehow incompatible: it is known that non-elementary finitely generated hyperbolic groups contain a free non-abelian subgroup $F_2$ and are thus non-amenable. On the other hand, in \cite{CT}, the authors prove the counter-intuitive fact that there {\em do} exist {\em amenable} non-elementary  locally compact hyperbolic groups! There are plenty of other differences as well: for example, it is possible for locally compact hyperbolic groups to act transitively on their boundary, even if the boundary is infinite.  In fact, such groups are studied in \cite{Caprace}, \cite{Carette-Dreesen}. On the other hand, non-elementary discrete hyperbolic groups are countable, and so cannot act transitively on their uncountable boundary.

The behaviour of groups with respect to embeddings into $L_p$-spaces $(p\geq 1)$ is directly related to properties such as the Haagerup property, Yu's property $A$, Kazhdan's property $T$, coarse embeddability and hence it is related to important conjectures such as the Novikov and Baum-Connes conjecture. 
Recently, a lot of work was done investigating the equivariant and non-equivariant embeddability behaviour of discrete hyperbolic groups into $L_p$-spaces $(p\geq 1)$ (see e.g. \cite{Tesseramain}, \cite{Brodskiy}, \cite{Yu:1}, \cite{Bourdonmain}, \cite{Nica} and others). For example, it was shown in \cite{Yu:1}, that although there are finitely generated hyperbolic groups with Kazhdan's property $(T)$ (which thus do not admit a proper affine isometric action on an $L_2$-space), each discrete hyperbolic group admits a proper affine isometric action on an $L_p$-space for $p$ sufficiently large! One of the things that we will show, is that this result persists in the locally compact hyperbolic setting. Moreover, we will be more specific and quantify {\em how proper} the action is, as this is related to group theoretic properties such as amenability and exactness. We give the necessary definitions before formulating our results.

\subsection{Equivariant and non-equivariant $L_p$-compression}
Throughout the remainder of this paper, and unless specifically mentioned otherwise, we denote by $p$ some fixed real number greater or equal to $1$.
\begin{defn}[see \cite{Gromov}]
A compactly generated group $G$ is {\bf coarsely embeddable into an $L_p$-space}, if there exists a measure space $(\Omega,\mu)$, a non-decreasing function $\rho_-:\R^+ \rightarrow \R^+$ such that $\lim_{t\to \infty} \rho_-(t)=+\infty$, a constant $C>0$ and a map $f:G\rightarrow L_p(\Omega,\mu)$, such that
\[ \rho_-(d(g,h)) \leq \|f(g)-f(h)\|_p \leq C d(g,h) \ \forall g,h\in G, \]
where $d$ is the word length metric relative to a compact generating subset.
The map $f$ is called a {\bf coarse embedding} of $G$ into $L_p(\Omega,\mu)$ and the map $\rho_-$ is called a {\bf compression function} for $f$.
\end{defn}
\begin{defn}
Let $G$ be a group and $(\Omega,\mu)$ a measure space. Fix $p\geq 1$. A map $f:G\rightarrow L_p(\Omega,\mu)$ is called {\bf $G$-equivariant}, if there is an affine isometric action $\alpha$ of $G$ on $L_p(\Omega,\mu)$ such that $\forall g,h\in G: f(gh)=\alpha(g)(f(h))$.
A locally compact second countable group $G$ admitting an equivariant coarse embedding into a Hilbert space is said to satisfy the {\bf Haagerup property}. 
\end{defn}
The Haagerup property is a subject of intense study \cite{Haagerup} and is known to imply the strong Baum-Connes conjecture. In $2004$, Guentner and Kaminker introduced two  numerical invariants \cite{guekam}. The first quantifies how well a group satisfies the Haagerup property, i.e. how well it embeds coarsely and equivariantly into an $L_2$-space. The other quantifies how well the group embeds coarsely, but not necessarily equivariantly, into an $L_2$-space.  The latter invariant links coarse embeddability to the well-studied notion of quasi-isometric embeddability \cite{CTV}.
\begin{defn}
Fix $p\geq 1$. Given a compactly generated locally compact group $G$ and a measure space $(\Omega,\mu)$, the {\bf $L_p$-compression} $R(f)$ of a coarse embedding $f:G\rightarrow L_p(\Omega,\mu)$ is defined as the supremum of $r\in [0,1]$ such that
\[ \exists C,D>0, \forall g,h\in G: \frac{1}{C} d(g,h)^r - D \leq \| f(g)-f(h)\| \leq C d(g,h) .\]
The {\bf equivariant $L_p$-compression} $\alpha_p^*(G)$ of $G$ is defined as the supremum of $R(f)$ taken over all $G$-equivariant coarse embeddings of $G$ into all possible $L_p$-spaces. Taking the supremum of $R(f)$ over all, also non-equivariant, coarse embeddings, leads to the (non-equivariant) {\bf $L_p$-compression} $\alpha_p(G)$ of $G$. One can show that both of these definitions do not depend on the chosen compact generating neighbourhood.
\end{defn}
Equivariant and non-equivariant compression are related to interesting group theoretic properties. 
Indeed, based on a remark by M. Gromov, it was shown that the equivariant and non-equivariant $L_2$-compression are equal for amenable groups (see \cite{CTV}). Moreover, if the equivariant $L_2$-compression of a compactly generated group is $>\frac{1}{2}$ then the group is amenable \cite{CTV}. It is straighforward to show this result for all locally compact second countable groups. The result provides a partial converse to the statement that amenable groups satisfy the Haagerup property. There is a similar statement for the equivariant $L_p$-compression for $p\neq 2$: in \cite{naoper}, the authors show that a finitely generated group is amenable if its equivariant $L_p$-compression is strictly greater than $\frac{1}{\min(2,p)}$. In the non-equivariant setting, we have the result that finitely generated groups with non-equivariant compression $>1/2$ satisfy non-equivariant amenability, i.e. property $A$ (which is equivalent to exactness of the reduced $C^*$-algebra) \cite{guekam}. In this paper, we generalize this result to the setting of locally compact second countable groups (Theorem \ref{theorem:PropertyA}).

For finitely generated groups, 
the non-equivariant Haagerup property, i.e. coarse embeddability into a Hilbert space, implies the Novikov conjecture. This result was suggested by Gromov in \cite{Gromov2}, but he did not give a proof. Around $2000$, in a breakthrough article, Yu showed that such groups satisfy the coarse Baum-Connes conjecture, which is related to the Novikov conjecture \cite{Yu}. Together with Skandalis and Tu, the Novikov conjecture for finitely generated coarsely embeddable groups was proven in \cite{Yu:2}. Later, in \cite{Yu:3}, the authors prove the same result for embeddings into uniformly convex Banach spaces and this is one of the motivations to study embeddings into $L_p$-spaces for $p\neq 2$.


\subsection{Results}
In \cite{Nowak}, P. Nowak gives new geometric characterizations of the Haagerup property by proving that a locally compact second countable group is Haagerup if and only if it admits a proper, affine isometric action on some $L_p([0,1])$ for $1<p<2$. This work was motivated by a question asked by A. Valette (see Section 7.4.2 in \cite{Haagerup}). In \cite{CDH}, the authors show that any group $G$ with the Haagerup property also admits a proper affine isometric action on any $L_p$-space $p\geq 1$. By analysing their methods, we find moreover the lower bound $\alpha_p^*(G)\geq \alpha_2^*(G)/p$ for all $p\geq 1$. Naor and Peres prove that, for finitely generated groups, the equivariant $L_2$-compression is actually minimal among all the other equivariant $L_p$-compressions $(p\geq 1)$ (Lemma $2.3$ in \cite{naoper}). We generalize this result to the locally compact context.
\begin{prop}[See Proposition \ref{prop:l2isminimal}]
For every locally compact, compactly generated group $G$, we have $\alpha^*_p(G)\geq \alpha^*_2(G)$ for all $p\geq 1$.
\label{theorem:2}
\end{prop}
This observation leads to the following corollaries.
%
\begin{corollary}
Let $n\geq 2$. For all $p\geq 1$, we have $\alpha_p^*(SO(n,1))=\max(1/2,1/p)$.
If $p\geq 2$, we have $\alpha_p^*(SU(n,1))=1/2$.
\label{corollary:5} \label{corollary:4}
\end{corollary}
\noindent We refer to Corollaries \ref{corollary:application1} and \ref{corollary:application1bis} in the text for the proofs.

Next, we study locally compact hyperbolic groups. It follows from \cite{BDS} or \cite{Bonk} that these groups embed quasi-isometrically into a finite product of binary trees. The results of R. Tessera \cite{Tesseramain} then imply that locally compact hyperbolic groups have $L_p$-compression equal to $1$ for every $p\geq 1$. Together with proposition \ref{theorem:2}, we obtain the following result which can be found as Corollary \ref{cor:amlochyp} later in this paper.
\begin{corollary}
If $G$ is an amenable locally compact hyperbolic group, then $\alpha_p^*(G)=1$ for every $p\geq 1$.
\label{cor:amco1}
\end{corollary}

After this, we elaborate on Yu's result that finitely generated hyperbolic groups admit proper affine isometric actions, with compression $\geq 1/p$, on an $L_p$-space for $p$ sufficiently large \cite{Yu:1} (this result also follows independently from both the work of B. Nica \cite{Nica} and the work of M. Bourdon \cite{Bourdonmain}). Bourdon shows moreover that $p$ is {\em sufficiently large} when it is strictly larger than the Ahlfors regular conformal dimension of the hyperbolic boundary of the group (see \cite{Bourdonmain}). Yves de Cornulier asked whether a similar result would be valid in the locally compact hyperbolic setting. We answer this question affirmatively. Note that we are allowed to exclude the amenable case by Corollary \ref{cor:amco1}.
\begin{theorem}[See Theorem \ref{theorem:conformal}]
Let $G$ be a non-amenable locally compact non-elementary hyperbolic group and denote the Ahlfors regular conformal dimension of $\partial G$ by $Q$. Then for each $p>Q$, there is a metrically proper affine isometric action of $G$ on a finite $l_p$-direct sum of copies of $L_p(G)$. The linear part of the action is the finite direct sum of the natural translation actions. One can choose the corresponding $1$-cocycle to have compression function $\succeq t^{1/p}$. In particular, we thus have $\alpha_p^*(G)\geq 1/p$. \label{theorem:3}
\end{theorem}
\begin{remark}
Bogdan Nica \cite{Nica} proved that finitely generated hyperbolic groups admit a proper affine isometric action on $L^p(\partial G \times \partial G)$ for $p$ {\em sufficiently large}. The advantage of this result is that the $L_p$-space on which the group acts, is explicitly known: it is $L^p(\partial G \times \partial G)$. The down side is that he does not obtain a nice lower bound on the minimal $p$ for which this is true. Bogdan Nica informed us that it should be possible to generalize his proof to the locally compact hyperbolic groups that contain a finitely generated hyperbolic and cocompact subgroup.
\end{remark}
Pierre Pansu showed that the Ahlfors regular conformal dimension of the boundary of $G=Sp(n,1)$ (and $F_4^{-20}$) equals $4n+2$ (and $22$ respectively) \cite{Pansu}. Based on the observation that $G$ is hyperbolic, our methods thus also provide a new proof for the following. 
\begin{corollary}
The group $G=F_4^{-20}$ for $p>22$ and the groups $G=Sp(n,1)$ for $p>4n+2$ admit proper affine isometric actions on an $L_p$-space such that the associated $1$-cocycle has compression at least $1/p$. 
\end{corollary}
In \cite{Valette}, the authors' main result is that $Sp(n,1)$ (and $F_4^{-20}$) admit a proper affine isometric action on $L_p(Sp(n,1))$ ($L_p(F_4^{-20})$ respectively) for $p>4n+2$ ($p>22$ respectively). The method in \cite{Valette} does not (and cannot) provide any lower bound for the compression of the $L_p$-cocycle.

Finally, we turn to the class of second countable, locally compact groups with property $(A)$. Similarly to the generalization of hyperbolicity from the discrete to the locally compact world, there has been growing interest in the notion of property $(A)$ for locally compact second countable groups (see e.g. \cite{DL1},\cite{DL2}). By a result of Guentner and Kaminker \cite{guekam}, it is known that a finitely generated group with compression $>1/2$ has property $(A)$. We generalize this result to all locally compact second countable groups in Theorem \ref{theorem:PropertyA}:
\begin{theorem}
 Let $G$ be a locally compact second countable group and let $d$ be a proper, left-invariant metric on $G$, generating the topology and with exponentially controlled growth of balls (see \ref{def:expcongro}). If $\alpha(G)>1/2$ then $G$ has property A.
\end{theorem}

\section{Preliminaries}
\subsection{Visual metrics}
\begin{convention}
Throughout this section, $G$ is a locally compact, hyperbolic group, equipped with the word length metric relative to a compact generating subset. Throughout this paper, $(X,d)$ will always denote a proper (i.e. closed balls are compact) hyperbolic geodesic metric space on which $G$ acts properly and cocompactly by isometries. We denote the hyperbolicity constant of $X$ by $\delta$. 
\end{convention}
Similarly to the finitely generated case, one defines the hyperbolic boundary $\partial G$ of $G$ as the hyperbolic boundary $\partial X$ of $(X,d)$, i.e. as the set of equivalence classes of geodesic rays in $X$ where two rays are said to be {\bf equivalent} if they lie at bounded Hausdorff distance from each other. Note that, as $G$ acts by isometries on $X$, and thus maps geodesic rays to geodesic rays, one obtains a natural $G$-action on $\partial X=\partial G$. It is well known that the boundary of $X$, and thus of $G$, is a compact, perfect set, metrizable by a {\em visual metric}. 
\begin{defn}
A metric $\rho$ on the hyperbolic boundary of a proper hyperbolic geodesic metric space $(X,d)$ is called {\bf visual} if 
\begin{enumerate}
\item $\rho$ induces the canonical boundary topology on $\partial X$
\item there exist $x_0\in X, C>0$ and $a>1$ such that
for any two distinct $\xi,\xi'\in \partial X$ and for any bi-infinite geodesic $\gamma$ in $X$ connecting $\xi$ to $\xi'$ and any $y\in \gamma$ with $d(x_0,y)=d(x_0,\gamma)$ we have
\[ \frac{1}{C} a^{-d(x_0,y)} \leq \rho(\xi,\xi')\leq C a^{-d(x_0,y)} .\]
\end{enumerate}
\label{defn:visualmetric}
We sometimes denote $\rho=d_a$ (or $d_{a,C}$) if we need to emphasize the value of $a$ (and/or C).
\end{defn}

It is a standard fact that two visual metrics $d_a,d_b$ on $\partial X$ are H\"older equivalent \cite{Kapovich}. Stronger even, there exist $C>0$ and $\alpha>0$ such that for any $\xi,\xi'\in \partial X:$
\[ d_b(\xi,\xi')^\alpha/C \leq d_a(\xi,\xi') \leq C d_b(\xi,\xi')^\alpha. \]

\subsection{Q.I.-embedded free subgroups of locally compact hyperbolic groups}
Given a locally compact hyperbolic group $G$, it follows from \cite{Gromov:1} that its boundary is either empty, finite or uncountable. In the latter case, we say that $G$ is {\bf non-elementary} hyperbolic. Following \cite{Gromov:1}, this case has two important subcases:
\begin{enumerate}
\item The action of $G$ on its boundary is called {\bf focal} if and only if $G$ fixes a point $\xi\in \partial X$. In this case, $\{\xi \}$ is the unique finite orbit of the action.
\item The action of $G$ on its boundary is called {\bf of general type} if it has no finite orbit.
\end{enumerate}
Lemma 5.2 and 5.3 in \cite{Caprace} show that a non-elementary locally compact hyperbolic group is amenable if and only if the action on its boundary is focal. An application of the ping-pong lemma (see \cite{Gromov:1}, 8.2.E, 8.2.F) yields the following result.
\begin{lemma}[Lemma 3.3 in \cite{Caprace}]
Every non-elementary locally compact hyperbolic group $G$ contains a quasi-isometrically embedded copy of the free semigroup on two generators. If the action of $G$ on its boundary is of general type, then G contains a quasi-isometrically embedded copy of the free group on two generators.
\label{lm:qiF2}
\end{lemma}
The focal case has been characterized and turns out to be quite restricted (see Theorem A in \cite{Caprace}).
Using the standard fact that $\alpha_p^*(F_2)=\max(1/2,1/p)$ for $p\geq 1$ (see \cite{naoper}), we immediately obtain the following corollary.
\begin{corollary} \label{cor:upp}
If $G$ is a non-elementary and non-amenable locally compact hyperbolic group, then $\alpha_2^*(G)\leq \max(1/2,1/p)$.
\end{corollary}

\section{The (equivariant) $L_2$-compression of locally compact compactly generated groups is minimal}
It is known that $L_2$ embeds isometrically into $L_p$ for any $p$ (see \cite{Benyamini}, page 189) so the non-equivariant $L_2$-compression is minimal among all $L_p$-compressions. The equivariant case is harder to solve.
\begin{remark}
A first approach to tackle the equivariant case could be to quantify the results in \cite{CDH}. This leads to the lower bound $\alpha_p^*(G)\geq \alpha_2^*(G)/2$ for all $p\geq 1$ and all compactly generated locally compact groups $G$. We give a brief overview of how we obtained this bound. For underlying definitions, we refer the reader to \cite{CDH}.

Given a proper affine isometric action of a group $G$ on a Hilbert space $\mathcal{H}$, let $b:G\rightarrow \mathcal{H}$ be the orbit map of $0\in \mathcal{H}$. Assume that there is some $\alpha>0$ such that $\|b(g)\|\geq \frac{1}{C} l(g)^\alpha$ for all $g\in G$ and where $l$ is the word length of $G$ with respect to a compact generating subset. It is well known that $\psi:(g,g')\mapsto \|b(g)-b(g')\|^2$ is a so called {\em conditionally negative definite kernel}. From Corollary 6.18 and Lemma 6.15 in \cite{CDH}, it follows that there is a {\em structure of space with measured walls $(X,W,B,\mu)$} and a point $x_0\in X$ such that $\mu(W(gx_0\mid g'x_0))=\sqrt{\psi(g,g')}\geq \frac{1}{C}l(g^{-1}g')^\alpha$, for all $g,g'\in G$. Next,
from their Lemma 3.10, we deduce that for every $p\geq 2$, we have $\alpha_p^*(G)\geq \alpha/p$. Taking the supremum over $\alpha$, we obtain $\alpha_p^*(G)\geq \alpha_2^*(G)/p$ for every $p\geq 1$.
\end{remark}
Naor and Peres' Lemma $2.3$ in \cite{naoper} states that for every finitely generated group $G$ and any $p\geq 1$, one actually has $\alpha^*_p(G)\geq \alpha^*_2(G)$. We will use the following lemma to generalize their result.
\begin{lemma}
For any compact $N\triangleleft G$, we have $\alpha^*_2(G/N)=\alpha^*_2(G)$.
\label{lm:quotient}
\end{lemma}
\begin{proof}
The inequality $\alpha^*_2(G/N)\leq \alpha^*_2(G)$ follows from the fact that any affine isometric action of $G/N$ on a Hilbert space gives rise to an affine isometric action of $G$ by precomposing with the quotient map $G \twoheadrightarrow G/N$. 
Conversely, to show that $\alpha^*_2(G/N)\geq \alpha^*_2(G)$, note that the action of the compact subgroup $N$ has a fixed point. Translating this point to coincide with the $0$-element of our Hilbert space and starting with a proper affine isometric action of $G$ on a Hilbert space, we obtain a proper affine isometric action of $G$ whose kernel contains $N$. The induced action of $G/N$ comes with a $1$-cocycle leading to the conclusion $\alpha^*_2(G/N)\geq \alpha^*_2(G)$.
\end{proof}
\begin{remark}
More generally, $\alpha^*_p(G/N)\leq \alpha^*_p(G)$. The converse inequality can not be generalized because our proof implicitly uses that a closed subspace of an $L_2$-space is again an $L_2$-space.
\end{remark}
%
\begin{prop}
For every locally compact, compactly generated group $G$, equipped with the word length distance relative to a compact generating subset, and for any $p\geq 1$, we have $\alpha^*_p(G)\geq \alpha^*_2(G)$.
\label{prop:l2isminimal}
\end{prop}
\begin{proof} 
It is well known (see \cite{KK}) that any locally compact, compactly generated group $G$ has a compact normal subgroup $N$ such that $G/N$ is separable. By Lemma \ref{lm:quotient}, we can thus assume that $G$ is separable and so for the purpose of calculating equivariant Hilbert space compression, we can assume that our Hilbert space is $l^2(\Z)$. 
This fact turns out to be sufficient in order for the proof of Naor and Peres (Lemma $2.3$ in \cite{naoper}) to hold and so we conclude. 
\end{proof}
Combining this with material on measured walls spaces as explained in the author's thesis \cite{Dreesenthesis}, we obtain the following corollaries.
\begin{corollary}
For $p\geq 2, n\geq 1$, we have $\alpha_p^*(SU(n,1))=1/2$.
\label{corollary:application1}
\end{corollary}
\begin{proof}
By Corollary \ref{cor:upp}, we have $\alpha_p^*(SU(n,1))\leq \max(1/2,1/p)$ for every $p\geq 1$. 
As in Example $5.3.5$ in \cite{Dreesenthesis}, one obtains the lower bound $\alpha_2^*(SU(n,1))\geq 1/2$. We use Proposition \ref{prop:l2isminimal} to conclude.
\end{proof}

%
We can prove a stronger result for $G=SO(n,1)$.
\begin{corollary}
For every $p\geq 1, n\geq 2$, we have $\alpha_p^*(SO(n,1))=\max(1/2,1/p)$.
\label{corollary:application1bis}
\end{corollary}
\begin{proof}
As in Example $5.3.5$ of \cite{Dreesenthesis}, we get that $\alpha_p^*(SO(n,1))\geq 1/p$. We conclude by combining this with Proposition \ref{prop:l2isminimal} and Corollary \ref{cor:upp}.

\end{proof}
\begin{remark}
We are unable to proof Corollary \ref{corollary:application1} for $1\leq p \leq 2$ because there is no known alternative to a result by Robertson, linking the (real) hyperbolic distance $d(x,y)$ to the measure of the collection of half-spaces containing $x$ but not $y$ (see Corollary 2.5 in \cite{Robertson} or Corollary 5.3.1 in \cite{Dreesenthesis}).
\end{remark}
Proposition \ref{prop:l2isminimal} also allows us to calculate the equivariant $L_p$-compressions of amenable hyperbolic groups. In fact, locally compact hyperbolic groups embed quasi-isometrically into a finite product of binary trees. This can be deduced from Theorem 1.2 in \cite{BDS} or from Theorem 1.1 in \cite{Bonk}, combined with the fact that hyperbolic $n$-space embeds isometrically into a product of binary trees. In \cite{Tesseramain} (see Theorems 3 and 4), R. Tessera (see \cite{Tesseramain}) studied the embeddability behaviour of trees into $L_p$-spaces. Using his results, it follows that $\alpha_p(G)=1$ for every $p\geq 1$ and every locally compact hyperbolic group $G$.

\begin{corollary}
If $G$ is an amenable locally compact hyperbolic group, then $\alpha_p^*(G)=1$ for every $p\geq 1$.
\label{cor:amlochyp}
\end{corollary}
\begin{proof}
As $G$ is amenable, we have that $\alpha_2^*(G)=\alpha_2(G)$ \cite{CTV}, hence $\alpha_2^*(G)=1$. Combining this with Proposition \ref{prop:l2isminimal}, we conclude that $\alpha_p^*(G)=1$ for every $p\geq 1$.
\end{proof}
\section{Proper affine isometric actions of locally compact hyperbolic groups on $L_p$-spaces: a first result \label{sc:easierproof}}
With respect to the Haagerup property, there exist discrete hyperbolic groups on both ends of the spectrum. Similarly, in the non-discrete, locally compact case, we have locally compact hyperbolic groups which are Haagerup, e.g. $SO(n,1)$ and others which have property $(T)$ and are thus far from being Haagerup, e.g. $Sp(n,1)$. Although locally compact hyperbolic groups that are not Haagerup do {\em not} admit a proper affine isometric action on any $L_2$-space, we now show that they {\em do} admit a proper affine isometric action on an $L_p$-space for $p$ sufficiently large. We will henceforth write $L_p(G)=L_p(G,\mu)$ where $\mu$ is the right Haar measure on $G$.
\begin{theorem}
Let $G$ be a non-amenable locally compact non-elementary hyperbolic group. Then $G$ admits a metrically proper affine isometric action on a finite direct sum of copies of $L_p(G)$ when $p$ is sufficiently large. The linear part of the action is the direct sum of the natural translation actions on the $L_p(G)$. One can choose the corresponding $1$-cocycle to have compression function $\succeq t^{1/p}$, so the compression of the associated $1$-cocyle is greater than $1/p$. \label{theorem:maineasy} \label{theorem:easierresult}
\end{theorem}
Our proof uses ideas from \cite{Bourdonmain} and consists of three parts. First, we formally define a map $c$ on $G$ which satisfies the $1$-cocycle relation with respect to the natural right translation action $\rho$ on $L_p(G,\mu)$. Here, the hardest part is to show that for some $p$ large enough, we have $c(g)\in L_p(G,\mu)$ for every $g\in G$. As a second step, we show that $c$ is also continuous: this is trivially true in the discrete case, but makes crucial use of Proposition 5.10 of \cite{Caprace} for the locally compact setting. The final step consists of showing that $c$ is also proper. Here, basically the proof is obtained by replacing sums by integrals in the proof for the discrete case, see Section $2.3$ in \cite{Bourdonmain}. We merely mention this proof for completeness.

The proof that we present gives a concrete number $N$, in terms of properties of $G$, such that $G$ admits a metrically proper affine isometric action on an $L_p$-space for all $p\geq N$. We refer the reader to Remark \ref{remark:valueofN} for details. The discussion on page 5 of \cite{Mackay} indicates that the lower bound $N$ obtained here is greater (i.e. less good) than the one we will obtain in Theorem \ref{theorem:maindifficult}. The key to improve the lower bound will be to strengthen Lemma \ref{lm:twerkt} below and we will do this in Section \ref{sc:harder}. We give the proof of Theorem \ref{theorem:maineasy} because it is easier than that of Theorem \ref{theorem:maindifficult}, because most parts can be reused in Section \ref{sc:harder}, because it requires no additional preliminaries and because the obtained lower bound $N$ is often easier to calculate than the Ahlfors regular conformal dimension.\\

\begin{proof}[\bf The proof of Theorem \ref{theorem:easierresult}.]
Let $X$ be a proper geodesic hyperbolic metric space admitting a proper cocompact $G$-action by isometries. From Proposition 5.10 in \cite{Caprace}, we conclude that modulo a compact normal subgroup $W<G$, we have that $G$ is totally disconnected or that it is the full isometry group (or its identity component) of a rank 1 symmetric space of noncompact type. 
So, in the former case, we can take $X$ to be a Schreier graph (the set of vertices is $\{\gamma V\mid \gamma \in G\}$ where $V$ is an open compact subgroup of $G$; see the remarks on the bottom of page 2 in \cite{Caprace} for further details). In the latter case we can take $X$ to be a rank 1 symmetric space of noncompact type. As $W\triangleleft G$ is compact, we have $\alpha_p^*(G)\geq \alpha_p^*(G/W)$ for every $p\geq 1$: indeed any affine isometric action of $G/W$ on an $L_p$-space gives rise to one of $G$ by precomposing with the projection map $G\twoheadrightarrow G/W$. We are thus allowed to henceforth assume $G:=G/W$.

As in Definition \ref{defn:visualmetric}, let $d_a$ be a visual metric on $\partial X$ and $x_0\in X$.
Choose a Lipschitz function $u$ on $\partial X$ (for concreteness, let us choose a base point $\xi_0\in \partial X$ and set $u:\xi \mapsto d_a(\xi_0,\xi)$). Define a function $f_u:G \rightarrow \R$ as follows: set $f_u:g \mapsto u(\xi_{g})$ where $\xi_{g}$ is any element in the boundary of $X$ corresponding to a geodesic ray starting in $x_0$ through $gx_0$. In the totally disconnected case, the stabilizer of $x_0$ contains an open subset and so $f_u$ is continuous. In the rank 1 case, $f_u$ is continuous outside of the stabilizer of $x_0$, which is a measure 0 set $S$. Hence, by modifying $f_u$ on a measure $0$ set, we can assume that $f_u$ is measurable.

Consider the standard right translation action $\rho$ on $L_p(G,\mu)$, i.e. for every $g\in G$ and $f\in L_p(G,\mu)$ we set $\rho(g)(f):G \rightarrow \R, \gamma \mapsto f(\gamma g)$. We formally define a $1$-cocycle with respect to $\rho$ as follows: 
\[ c:=g \in G \mapsto \rho(g)(f_u)-f_u .\]
We first show that $c(g)\in L_p(G,\mu)$ for every $g \in G$ and $p$ sufficiently large. 

\begin{lemma}
For $p$ sufficiently large, we have $\forall g\in G: \ c(g)\in L_p(G,\mu)$.
\label{lm:twerkt}
\end{lemma}
\begin{proof}
Let $a,C\in \R$ denote constants as in Definition \ref{defn:visualmetric}. 
Let $K\subset G$ be a compact neighbourhood of $1$ that generates $G$. Consider the cover $(K\gamma)_{\gamma\in K^2}$ of $K^2$. Because of compactness, we can derive a finite subcover of $K^2$, say $K\gamma_1,K\gamma_2,\ldots ,K\gamma_n$. It is easy to see that $K^3$ can be covered by the $(K\gamma_i \gamma_j)_{i,j=1,2,\ldots n}$. Similarly, 
$K^m$ is covered by $n^{m-1}$ right translates of $K$.

Fixing $g\in G$ and neglecting the measure $0$ set $S$ mentioned earlier, we write
\begin{eqnarray*}
\| c(g) \|_p^p & = &
\sum_{i=1}^\infty \int_{K^i\backslash K^{i-1}} \lvert f_u(\gamma g)-f_u(\gamma) \rvert^p d\mu \\
& \leq & \sum_{i=1}^\infty \max_{\gamma \in K^i\backslash K^{i-1}} (\lvert f_u(\gamma g)-f_u(\gamma) \rvert^p) \mu(K^i\backslash K^{i-1}) \\
& \leq & \sum_{i=1}^\infty \max_{\gamma \in K^i\backslash K^{i-1}} (\lvert f_u(\gamma g)-f_u(\gamma) \rvert^p) n^{i-1} \mu(K).
\end{eqnarray*} 
By definition,
\[ \lvert f_u(\gamma g)-f_u(\gamma)\rvert = \lvert u(\xi_{\gamma g})-u(\xi_\gamma) \rvert \leq  d_a(\xi_{\gamma g},\xi_{\gamma}) \leq  C a^{-d(x_0,y_0)},\]
where $y_0\in (\xi_{\gamma g},\xi_{\gamma})\subset X$ lies at minimal distance from $x_0$. In order to bound $\lvert f_u(\gamma g)-f_u(\gamma)\rvert$, we look for a lower bound on $d(x_0,y_0)$.

For $h\in \{\gamma g, \gamma\}$, we let $r_h$ 
be a geodesic ray in the class of $\xi_h$ which starts in $x_0$ and and goes through $hx_0$. If $x_h^i$ 
is the point on $(hx_0,\xi_h)$ at distance $i$ from $hx_0$, then a bi-infinite geodesic $(\xi_{\gamma g},\xi_{\gamma})$ can be obtained as 
the limit of geodesics $[x_\gamma^i,x_{\gamma g}^i]$ as $i\to \infty$. These
geodesics do not pass through the ball of radius 
\begin{eqnarray*}
\min(d(x_0,\gamma g x_0),d(x_0,\gamma x_0)) -d(\gamma gx_0,\gamma x_0)/2 &=& \min(d(x_0,\gamma g x_0),d(x_0,\gamma x_0))\\
& &  -d(gx_0,x_0)/2.
\end{eqnarray*}
around $x_0$. As $d(x_0,\gamma g x_0)\geq d(x_0, \gamma x_0) - d(x_0, g x_0 )$, the geodesics do not pass through the ball of radius $d(x_0,\gamma x_0) -(3/2) d(x_0, gx_0)$ around $x_0$. Consequently,
\[ \lvert f_u(\gamma g)-f_u(\gamma)\rvert \leq C a^{-d(x_0,y_0)} \leq C a^{-[d(x_0,\gamma x_0) -3/2 d(x_0,gx_0)]}:= \overline{C} a^{-d(x_0,\gamma x_0)}, \]
where we set $\overline{C}=Ca^{(3/2) d(x_0,gx_0)}$.
 As the map $G \rightarrow X, g \mapsto g x_0$ is a quasi-isometry, we can choose constants $A>0, B\geq 0$ such that
\begin{equation}
\forall g_1,g_2\in G:\ \frac{1}{A} d_G(g_1,g_2) -B \leq d_X(g_1 x_0, g_2 x_0) \leq A d_G(g_1,g_2)+B.
\label{eq:rescaling}
\end{equation}
As $d(x_0,\gamma x_0)\geq \frac{1}{A} \lvert \gamma \rvert-B$, we obtain that
\[ \lvert f_u(\gamma g)-f_u(\gamma)\rvert \leq \widetilde{C} a^{-\frac{\lvert \gamma \rvert}{A}}, \]
where $\widetilde{C}=\overline{C}a^B$
We thus have
\begin{equation}
\|c(g)\|_p^p \leq \widetilde{C} \mu(K) \sum_{i=1}^\infty n^{i-1}a^{\frac{-p}{A}(i-1)}= \widetilde{C} \mu(K) \sum_{i=1}^\infty (na^{\frac{-p}{A}})^{i-1}.
\label{eq:converges}
\end{equation}
Taking $p$ such that $na^{\frac{-p}{A}}<1$, we get that $\|c(g)\|_p<\infty$, hence that $c(g)\in L_p(G,\mu)$ for every $g\in G$.
\end{proof}
Next, we show that $c$ is continuous.
\begin{lemma}
Assume that $p$ is such that $c(g)\in L_p(G,\mu)$ for every $g\in G$. Then, the map $c:G\rightarrow L_p(G,\mu)$ is continuous.
\label{lm:continuouscocycle}
\end{lemma}
\begin{proof}
As $\|c(g)-c(h)\|=\|c(h^{-1}g)-c(1)\|$, it suffices to show that $c$ is continuous in $1\in G$. So, fixing $\epsilon$, we need to find a neighbourhood $1\in V\subset G$ such that $\|c(v)\|\leq \epsilon$ for every $v\in V$. In the case where $G$ is totally disconnected, this is trivial: just take $V$ inside the stabilizer of $x_0$. Let us thus focus on the rank 1 case.
By Equation (\ref{eq:converges}) in the previous lemma, and the fact that this infinite sum must converge, we can find an open $1\in V_1\subset G$ and $I$ large enough such that 
\begin{enumerate}
\item $\widetilde{C} \mu(K) \sum_{i=I}^\infty (na^{\frac{-p}{A}})^{i-1}\leq \epsilon/2$ for every $v\in V_1$ (note that $\widetilde{C}$ depends on $v\in V_1$),
\item $K^{I-1} \supset \mbox{Stab}(x_0)V_1$.
\end{enumerate}
The problem is now reduced to choosing $V\subset V_1$ small enough such that one can bound the norm of $c(v), v\in V$ restricted to $K^{I-1}$. Take first $V\subset V_1$ with $V=V^{-1}$ such that $\mu(\mbox{stab}(x_0)V)<\frac{\epsilon}{4\mbox{diam}(\partial X)}$. One can then bound the norm of $c(v)$ restricted to ${K^{I-1}}$ by $\epsilon/4$ plus the norm of $c(v)$ restricted to $K^{I-1}\backslash \mbox{stab}(x_0)V$. By decreasing $V$ further if necessary, we can assure that the distance between $\gamma v$ and $\gamma$ becomes arbitrarily small. Moreover, for $\gamma \in K^{I-1}\backslash \mbox{stab}(x_0)V$, nor $\gamma$, nor $\gamma v$ lies in stab$(x_0)$, so by continuity of $f_u$ we have $\lvert f_u(\gamma v)- f_u(\gamma) \rvert$ arbitrarily small. As continuity on a compact set implies uniform continuity, this leads to the conclusion. 
%
%
\end{proof}
Finally, we show that $c$ is also proper. The proof that we give below mainly follows from the proof in the discrete case by replacing sums by integrals, see Section $2.3$ in \cite{Bourdonmain}. We decide to give it for completeness.
\begin{lemma}
The $1$-cocycle $c$ is proper.
\label{lm:properness}
\end{lemma}
\begin{proof}
Choose $g\in G$ and denote the length of $g$ by $m$. Find elements $k_1,k_2,\ldots k_m\in K$ whose product equals $g$. Define $\gamma_i=(k_1k_2\ldots k_{3i})^{-1}$ for $i=1,2,\ldots ,\lfloor m/3 \rfloor$, where $\lfloor m/3 \rfloor$ is the largest integer smaller or equal to $m/3$. Write $A=\cup K\gamma_i=\sqcup K\gamma_i$. We obtain
\begin{eqnarray*}
 \|c(g)\|_p^p 
 & \geq & \int_A \lvert f_u(\gamma g)-f_u(\gamma) \rvert^p d\mu \\
 & \geq &
\sum_{i=1}^{\lfloor m/3 \rfloor} \inf_{\gamma \in K\gamma_i} \lvert u(\xi_{\gamma g})- u(\xi_\gamma) \rvert^p \mu(K)
 \end{eqnarray*}
For each $k\in K$, take the shortest path $k\gamma_i g, k\gamma_i g k_m^{-1}, k\gamma_i g k_m^{-1} k_{m-1}^{-1}, \ldots, k\gamma_i$. It is clear that this path passes through $k$ and hence through $K$. As $G \rightarrow X, \gamma \mapsto \gamma x_0$ is a quasi-isometry and as $X$ is hyperbolic, we conclude that the geodesics $[k\gamma_i g x_0,k\gamma_i x_0]$ pass through the ball of radius $R$ and 
centre $x_0$ for some $R>0$ independent of $g$. Using a hyperbolicity argument, we see that there exists $k_0\in \N$ and 
$C>0$ such that for $g$ large enough and for $i\in \{k_0,k_0+1,\ldots \lfloor m/3 \rfloor -k_0 \}$, the points $\xi_{k\gamma_i g}=\xi_{\gamma g}$ and $\xi_{k\gamma_i}=\xi{\gamma}$ are $C$-diametrically opposite, i.e. that $(\xi_{\gamma g},\xi_{\gamma})$ passes through the ball $B(x_0,C)$.

%
If we can find a constant $C'$ and choose $u$ such that $C$-diametrically opposite points $\xi,\eta \in \partial G$ satisfy $\lvert u(\xi)-u(\eta) \rvert \geq C'/(\mu(K))^{1/p}$, then we can conclude that when $\lvert g \rvert$ is sufficiently large:
\[ \|c(g) \|_p \geq ((\lvert g \rvert-2)/3 -2k_0)^{1/p}C', \]
i.e. that $\alpha^*_p(G)\geq \frac{1}{p}$.

The map $u:\partial G \rightarrow \R, \xi \mapsto d_a(\xi,\xi_0)$ that we have chosen may not satisfy the above condition. We may however always find a finite collection of Lipschitz functions $u_1,u_2,\ldots,u_s$ on $\partial G$ and a constant $C'>0$ such that for any $C$-diametrically opposite points $\chi,\eta \in \partial G$ there is some $u_i \ (i=1,2,\ldots ,s)$ such that $\lvert u_i(\chi)-u_i(\eta) \rvert \geq C'/(\mu(K))^{1/p}$. Indeed,  as we are using a visual metric on $\partial G$, we obtain a constant $\widetilde{C}$ such that $d_a(\xi,\eta)\geq \widetilde{C}$ for any $C$-diametrically opposite points $\xi,\eta \in \partial G$. Now cover $\partial G$ with open balls of radius $\overline{C}:=\widetilde{C}/3$ and derive a finite subcover $B(\xi_1,\overline{C}),B(\xi_2,\overline{C}), \ldots, B(\xi_s,\overline{C})$. For each $i$, define a Lipschitz map $u_i:\partial G \rightarrow \R, \xi \mapsto d_a(\xi,\xi_i)$. Given $C$-diametrically opposite points $\chi,\eta \in \partial G$, choose $\xi_j$ such that $d_a(\chi,\xi_j)\leq \overline{C}$.
As $d_a(\chi,\eta)\geq 3\overline{C}$, we obtain by the triangle inequality that $\lvert u_i(\chi)-u_i(\eta) \rvert \geq \overline{C}$.

 One can now conclude by reasoning as above where we consider the diagonal action of $G$ on $L_p(\sqcup_{i=1}^s G)$ instead of $L_p(G)$ and the cocycle $c=(c_1,c_2,\ldots ,c_s)$ where $c_i$ is associated to $u_i$ as before.
%
%
\end{proof}
\end{proof}
\begin{remark}
Assume that $d_a$ is a visual metric on $\partial X$, where $X$ is a proper hyperbolic geodesic metric space that admits a proper, cocompact isometric $G$-action. Take $A>0,B\geq 0$ such that $G\rightarrow X, g \mapsto g x_0$ is an $(A,B)$-quasi-isometry. Take $K$ a compact generating subset of $G$ and let $n\in \N$ be such that $K^2$ is contained in a union of $n$ right translates of $K$. Then the proof of Lemma \ref{lm:twerkt} shows that $p$ is {\em sufficiently large} whenever $p>A \frac{\ln(n)}{\ln(a)}$.
\label{remark:valueofN}
\end{remark}
\section{Proper affine isometric actions of locally compact hyperbolic groups on $L_p$-spaces: the improved result \label{sc:harder}}
In this section we improve the bound given in Remark \ref{remark:valueofN}. Our new bound will be the {\em Ahlfors regular conformal dimension} of the hyperbolic boundary of $G$, as introduced in Bourdon-Pajot \cite{BourdonPajot}.

\subsection{Preliminaries}
\subsubsection{Ahlfors regular conformal dimension}
Apart from visual metrics on the boundary of a hyperbolic group, one can also consider the class of Ahlfors regular metrics. The metric $d$ on a compact metric space $(Z,d)$ is called {\bf Ahlfors regular} if its Hausdorff dimension $Q$ and its $Q$-Hausdorff measure $\nu$ satisfy
\[ \frac{1}{C} r^Q \leq \nu(B(r)) \leq Cr^Q, \]
for every ball $B(r)\subset (Z,d)$ with $r\leq \mbox{diam}(Z,d)$ and where $C>0$ is a constant that is independent of $r$.
Any compact metric space $(Z,d)$ with $d$ Ahlfors regular satisfies the following properties. 
\begin{defn}
A compact metric space $(Z,d)$ is {\bf uniformly perfect} if there exists a constant $C>1$ such that for any ball $B(x,r)\subset Z$ with $0<r\leq \mbox{diam}(Z)$, we have $B(x,r)\setminus B(x,r/C)\neq \emptyset$. A compact metric space $(Z,d)$ is called {\bf doubling} if there exists a constant $C\geq 1$ such that for each $r>0$, each ball of radius $r$ can be covered by at most $C$ balls of radius $r/2$. Related to this, a measure $\nu$ on $(Z,d)$ is called {\bf doubling} if open balls all have finite measure and if there exists a constant $C$ such that for each $r>0$ and $z\in Z$, we have $\nu(B(z,2r))\leq C \nu(B(z,r))$. It is easy to check that the $Q$-Hausdorff measure $\nu$ of an Ahlfors regular (hence doubling) metric space of Hausdorff dimension $Q$, is doubling.
\end{defn}
Corollary 14.15 in \cite{Heinonen} shows that conversely, if $(Z,d)$ satisfies these properties (i.e. is uniformly perfect and has a doubling measure), then it can be equipped with an Ahlfors regular metric {\em quasi-symmetric} to $d$.
\begin{defn}
Two metrics $d,d'$ on a compact metric space $(Z,d)$ are {\bf quasi-symmetric} if they induce the same topology and if there exists a homeomorphism $\eta:(0,\infty) \rightarrow (0,\infty)$ such that for any $x,y,z\in Z$ and $t\in [0,\infty)$, we have
\[ \frac{d(x,z)}{d(y,z)}\leq t \Rightarrow  \frac{d'(x,z)}{d'(y,z)}\leq \eta(t) .\]
We note that the previous implication also implies
\[ \frac{d'(x,z)}{d'(y,z)}\leq t \Rightarrow \frac{d(x,z)}{d(y,z)}\leq \psi(t):=1/\eta^{-1}(t^{-1}) .\]
\end{defn}
As an example of an Ahlfors regular metric, we can take any visual metric on the boundary of a discrete hyperbolic group. More generally, given a locally compact hyperbolic group $G$, it acts properly cocompactly isometrically on a proper hyperbolic geodesic metric space $X$. The boundary of $G$ coincides with that of $X$ by definition. One can now show (by verifying the conditions in \cite{BourdonKleiner}, bottom of page 2) that the boundary of a locally compact hyperbolic group, equipped with a visual metric, is uniformly perfect and doubling. As the boundary is also a complete metric space, it carries a doubling measure (Theorem 13.3 in \cite{Heinonen}). In particular, any visual metric is quasi-symmetric to an Ahlfors regular metric.

It is often convenient to work with different characterizations of quasi-symmetric metrics. We state the following.
\begin{prop}
Let $(Z,d)$ be a compact, uniformly perfect metric space and assume that $d'$ is a metric on $Z$ that induces the same topology as $d$. Consider the following three conditions.
\begin{enumerate}
\item $d$ and $d'$ are quasi-symmetric.
\item $d$ is H\"older equivalent to $d'$, i.e. there exists $\alpha,\beta,C>0$ such that for all $x,y\in Z$, we have
\[ \frac{1}{C} d'(x,y)^\beta \leq d(x,y) \leq C d'(x,y)^\alpha .\]
\item $d'$ is annulus semi-quasiconformal with respect to $d$. This means that if we define $\overline{d}_{d'}(x,r)=\sup_{y \in B_{d'}(x,r)} d(x,y)$, then for any $\epsilon \in (0,1)$, there exists $\delta \in (0,1)$ such that a $d'$-annulus $B_{d'}(x,r)\setminus B_{d'}(x,\epsilon r)$ is contained in a $d$-annulus $B_d(x,(1+a)\overline{d}_{d'}(x,r))\setminus B_d(x,\delta \overline{d}_{d'}(x,r))$ for any $a>0$.
\end{enumerate}
Then, $1\rightarrow 2$ and $1\leftrightarrow 3$.
\label{prop:equivdef}
\end{prop}
\begin{proof}
The implication $(1) \rightarrow (2)$ follows from Remark 1.1.b in \cite{Bourdonholder}. A proof of the equivalence $(1)\leftrightarrow (3)$ can be found in \cite{Kigami}, beginning of part 2.
\end{proof}
\begin{defn}
Let $(Z,d)$ be a compact, uniformly perfect, doubling metric space. One defines the Ahlfors regular conformal gauge $\mathcal{J}_{AR}(Z,d)$ of $(Z,d)$ as
\[ \mathcal{J}_{AR}(Z,d)=\{ d'\mid d' \mbox{ is an Ahlfors regular metric on }Z\mbox{ quasi-symmetric to d}\} .\]
Visual metrics on the boundary of a proper hyperbolic metric space are always quasi-symmetric. We may thus abbreviate $\mathcal{J}_{AR}(Z,d)$ by $\mathcal{J}_{AR}(\partial G)$ if $(Z,d)$ is the boundary of a locally compact hyperbolic group $G$, equipped with a visual metric.
\end{defn}
\begin{defn}
Let $(Z,d)$ be a compact, uniformly perfect, doubling metric space. The Ahlfors regular conformal dimension, Confdim$(Z,d)$, of $(Z,d)$ is defined as
\[ \mbox{Confdim}(Z,d)=\inf\{\mbox{Hausdim}(Z,d')\mid d'\in \mathcal{J}_{\mbox{AR}}(Z,d)\},\]
where Hausdim$(Z,d')$ denotes the Hausdorff dimension of $(Z,d')$. Again, if $(Z,d)$ is the hyperbolic boundary of a hyperbolic group $G$, equipped with a visual metric, then we write Confdim$(Z,d)=$ Confdim$(\partial G)$.
\end{defn}
\subsubsection{Shadows}
Given a ball in $X$, say of radius $R$ and centre $x$, we define its shadow by
\[ \Sh(x,R)=\{\xi \in \partial X \mid \exists \mbox{ a geodesic ray } r=(x_0,\xi) \mbox{ such that } r\cap B(x,R)\neq \emptyset \} .\]
We state the following lemma for reference sake, omitting the proof which is not so difficult.
\begin{lemma}
Let $G$ be a locally compact hyperbolic group and $X$ an associated hyperbolic geodesic metric space.
If $\partial X \neq \emptyset$, then there exists $R>0$ such that $\Sh(x,R)\neq \emptyset$ for every $x\in X$.
\label{lm:fixedR}
\end{lemma}
In this paragraph, we elaborate on the relationship between shadows and balls. The underlying idea is that shadows ressemble $d$-balls whenever $d$ is quasi-symmetric to a visual metric.
The following observation will be used later, but the proof is actually an exercise. That is why we decide to omit it.
\begin{lemma}
Let $x_0\in X$ be some base-point, fix some $\xi \in \partial G$ and take $R$ as in Lemma \ref{lm:fixedR}. Assume that $A\subset \partial G$ is an open neighbourhood of $\xi$. Then, any geodesic ray in the class of $\xi$ emanating from $x_0$ contains a point $x$ such that $\Sh(x,2R)\subset A$.
\label{lm:shadowinball}
\end{lemma}
We also require the following lemma.
\begin{lemma}
Choose $x_0\in X$ and $R>\max(1,20\delta)$ such that for every $g\in G: S(gx_0,R)\neq \emptyset$.
Let $d$ be a metric in $\mathcal{J}_{\mbox{AR}}(\partial G)$. There exists a constant $D>0$ such that for all $g \in G$ there exists $r>0$ and $\xi\in \partial G$ with
\[ B_d(\xi,r) \subset \Sh(g x_0,2R) \subset B_d(\xi,Dr).\]
\label{lm:shadowsballs}
\end{lemma}
\begin{proof}
It is an exercise in hyperbolic geometry to verify the above result for a visual metric $d_a$. 
If $d$ is {\em any} metric in $\mathcal{J}_{\mbox{AR}}(\partial G)$, then it must be quasi-symmetric to $d_a$. By Proposition \ref{prop:equivdef}, $d$ is then annulus semi-quasiconformal to $d_a$, hence the result follows.
\end{proof}
\subsection{A better bound}
In this paragraph, we prove the following result.
\begin{theorem}	
Let $G$ be a non-amenable locally compact non-elementary hyperbolic group and denote the Ahlfors regular conformal dimension of $\partial G$ by $Q$. Then for each $p> Q$, there is a metrically proper affine isometric action of $G$ on a finite $l_p$-direct sum of copies of $L_p(G)$. The linear part of the action is the finite direct sum of the natural translation actions. One can choose the corresponding $1$-cocycle to have compression function $\succeq t^{1/p}$. In particular, for every $p>Q$, we have $\alpha_p^*(G)\geq 1/p$.
 \label{theorem:conformal} \label{theorem:maindifficult}
\end{theorem}
\begin{remark}
In the proof below, we will define the same $1$-cocycle $c$ as in the proof of Theorem \ref{theorem:easierresult}, hence, we will get the properness and continuity of $c$ for granted by the Lemmas \ref{lm:continuouscocycle} and \ref{lm:properness}. The main difficulty will be to show that $c(g)\in L_p(G,\mu)$ for every $g\in G$ and $p>Q$. Our proof uses ideas from Proposition $2.3$ in \cite{Bourdonmain}.
\end{remark}
\begin{proof}
This proof uses similar notations as in Section \ref{sc:easierproof}. So, let $(X,d_X)$ be a proper $\delta$-hyperbolic geodesic metric space that admits a proper cocompact isometric $G$-action. Let $K$ be a compact generating neighbourhood of $1\in G$.
Choose $x_0\in X, \xi_0\in \partial G$ and $R>\max(\mbox{diam}(Kx_0),20\delta)$.
Take $A>0,B\geq 0$ such that $G \rightarrow X, \gamma \mapsto \gamma x_0$ is an $(A,B)$-quasi-isometry. An open ball in $X$ with centre $x$ and radius $r$ will be denoted by $B_X(x,r)$. Let $d\in \mathcal{J}_{AR}(\partial G)$ and denote the Hausdorff measure and Hausdorff dimension of $(\partial G,d)=(\partial X,d)$ by $\nu$ and $\overline{Q}$ respectively.
 
As before, we define formally the $1$-cocycle $c:=g \in G \mapsto \rho(g)(f_u)-f_u$ with respect to the standard right translation action $\rho$ on $L_p(G,\mu)$. 
Here, as before, $u$ is a Lipschitz-map on $\partial G$ and $f_u:G \rightarrow \R$ maps $g\in G$ to $u(\xi_{g})$ where $\xi_{g}$ is any element in the boundary of $X$ corresponding to a geodesic ray starting in $x_0$ through $gx_0$.
We want to show that $c(g)\in L_p(G,\mu)$ for all $g\in G$ and $p>\overline{Q}$.

Choose $g\in G$ and write $g=k_1k_2 \ldots k_m$ where $k_i\in K$ and $m=\lvert g \rvert$ is the length of $G$. As one can find a constant $C(m)$ such that
\[ \|c(g)\| \leq C(m) \sum_{i=1}^m \| c(k_i) \|_p^p,\]
it suffices to consider the case $g=k\in K$.

 For each $\gamma \in G$, let us denote the minimal radius of a $d$-ball in $\partial G$ containing $\Sh(\gamma x_0,2R)$ by $r(\gamma)$. Note that $\Sh(\gamma x_0,2R)\cap \Sh(\gamma k x_0,2R)\neq \emptyset$, so $d(\xi_{\gamma k},\xi_\gamma)\leq 2(r(\gamma k)+r(\gamma))$. Using the fact that $u$ is Lipschitz and that $\mu$ is right-invariant, we obtain a constant $C_1>1$ such that
\[ \|c(k)\|_p^p \leq C_1 \int_G r(\gamma)^p d\mu. \]

By Lemma \ref{lm:shadowsballs}, there is $i\in \N$ such that for each $\gamma \in G$, there is $\xi \in \partial G$, such that $B(\xi,r(\gamma)/2^i)\subset \Sh(\gamma x_0,2R) \subset B(\xi,r(\gamma))$. So, by Ahlfors regularity (and the fact that Ahlfors regular measures are doubling), we see that for some $\xi\in \partial G$, we have $r(\gamma)^{\overline{Q}} \leq  C_2 C_3^i \nu(\Sh(\gamma x_0,2R))$,
where $C_2$ and $C_3$ are the constants appearing in the definition of Ahlfors regular metric and doubling measure respectively. Writing $C_2:=C_1(C_2C_3^i)^{p/\overline{Q}}$, we thus obtain
\begin{equation}
 \|c(k)\|_p^p \leq  C_2 \int_G (\nu(\Sh(\gamma x_0,2R)))^{p/\overline{Q}} d\mu
\label{eq:zoalsdiscreet}
\end{equation}

{\em Let us first bound the above integral by a suitably chosen infinite sum which we can later bound from above.}
To this end, 
write $V=K^{A(8R+2\delta+1+B)}, \epsilon_i=\frac{1}{2^i}\ (i\in \N)$ and choose $\gamma_0\in G$ such that $\nu(\Sh(\gamma_0x_0,2R))\geq \sup_{\gamma \in G} \nu(\Sh(\gamma x_0,2R))-\epsilon_0$.
As a next step, choose $\gamma_1 \in G\setminus(\gamma_0V)$ such that 
\[ \nu(\Sh(\gamma_1x_0,2R))\geq \sup_{\gamma \in G \setminus \gamma_0V} \nu(\Sh(\gamma x_0,2R))-\epsilon_1\]  and continue in this manner. So, choose $\gamma_2 \in G\setminus(\gamma_0V\cup \gamma_1V)$ such that 
\[\nu(\Sh(\gamma_2x_0,2R))\geq \sup_{\gamma \in G \setminus (\gamma_0V \cup \gamma_1V)} \nu(\Sh(\gamma x_0,2R))-\epsilon_2,...\] We claim that we so obtain a countable cover $\gamma_i V$ of $G$. To show this, assume by contradiction that $\cup_{i\in \N} \gamma_iV \neq G$. Using Lemma \ref{lm:shadowsballs}, we see that shadows $\Sh(y,2R)$ with $y\in X$ contain open balls of positive radius and hence have strictly positive measure by Ahlfors regularity. Given $\widetilde{\gamma}\in G\backslash \cup_{i\in \N} \gamma_iV$, we thus find $\widetilde{\epsilon}>0$ such that $\nu(\Sh(\widetilde{\gamma}x_0,2R))=\widetilde{\epsilon}$. By construction of the sequence $(\gamma_i)_i$, this thus means that $\nu(\Sh(\gamma_ix_0,2R))>\widetilde{\epsilon}-\epsilon_i>\widetilde{\epsilon}/2$ for $i$ sufficiently large. 
From the fact that $d$ is  H\"older equivalent with a visual metric, we can find a compact neighbourhood $L\subset G$ of $1$ such that $\nu(\Sh(\gamma x_0,2R))<\widetilde{\epsilon}/2$ whenever $\gamma \notin L$. In particular, infinitely many of the $\gamma_i$ must lie in $L$. By construction, these $\gamma_i$ have no limit point, contradicting compactness of $L$. We thus conclude that the $(\gamma_i V)_{i\in \N}$ form a countable cover of $G$ by left translates of $V$ such that the word length distance $d_V(\gamma_i,\gamma_j)\geq 2$ for all $i\neq j \in \N$.

Denote $\mathcal{G}=\{\gamma_i \mid i\in \N\}$ and for each $n\in \N$, define
\[ \mathcal{A}(n)=\{\gamma \in G \mid \gamma x_0 \in B_X(x_0,n)\setminus B_X(x_0,n-1) \}. \]
We can find constants $D_1,C_3>0$ such that
\begin{eqnarray*}
\|c(k)\|_p^p & \leq &
C_2 \sum_{n\in \N} \sum_{\gamma_i \in \mathcal{G}\cap \mathcal{A}(n)} (\nu(\Sh(\gamma_i x_0,2R))+\epsilon_i)^{p/\overline{Q}} \mu(V) \\
&\leq &
 D_1+ C_3 \sum_{n\in \N} \sum_{\gamma_i \in \mathcal{G}\cap \mathcal{A}(n)} \nu(\Sh(\gamma_i x_0,2R))^{\frac{p-\overline{Q}}{\overline{Q}}} \nu(\Sh(\gamma_i x_0,2R)) \mu(V).
\end{eqnarray*}

{\em The above infinite sum is chosen such that we can easily bound it from above. We start by bounding the factor $(\nu(\Sh(\gamma_i 
x_0,2R)))^{p-\overline{Q}/\overline{Q}}$ as in the discrete case.} Recall that $\Sh(\gamma_i x_0,2R)$ is contained in a ball of radius 
$r(\gamma)$, so by Ahlfors regularity, there is a constant $C_4>0$ such that $\nu(\Sh(\gamma_i x_0,2R)) \leq C_4 
r(\gamma_i)^{\overline{Q}}$. If $\widetilde{d}$ is some visual metric on the boundary, then there are constants $\widetilde{C_5}, 
\widetilde{b}>1$ such that
$\widetilde{r(\gamma)}\leq \widetilde{C_5}\widetilde{b}^{-d_X(x_0,\gamma x_0)}$, where $\widetilde{r(\gamma)}$ is the minimal radius of a $\widetilde{d}$-ball in $\partial X$ that contains $S(\gamma x_0,2R)$.
As $d\in \mathcal{J}_{\mbox{AR}}(\partial G)$, it is H\"older equivalent to a visual metric. So, there are constants $\overline{C_5},r>1$ such that $r(\gamma)\leq \overline{C_5} \widetilde{r(\gamma)}^r$. We then have
\[ r(\gamma) \leq \overline{C_5} \widetilde{r(\gamma)}^r \leq \overline{C_5} \widetilde{C_5}^r (\widetilde{b}^r)^{-d_X(x_0,\gamma x_0)} = C_5 b^{-d_X(x_0,\gamma x_0)}, \]
where $C_5:=\overline{C_5} \widetilde{C_5}^r$ and $b:=\widetilde{b}^r$.
This implies 
\[ \nu(\Sh(\gamma_i x_0,2R))^{(p-\overline{Q})/\overline{Q}} \leq C_4^{(p-\overline{Q})/\overline{Q}} (r(\gamma_i))^{p-\overline{Q}} \leq C_4^{(p-\overline{Q})/\overline{Q}} C_5^{p-\overline{Q}} b^{-d_X(x_0,\gamma x_0)(p-\overline{Q})} .\] 
Writing $C_4:= C_3C_4^{(p-\overline{Q})/\overline{Q}} C_5^{p-\overline{Q}}$, we thus obtain
\[ \|c(k)\|_p^p \leq  D_1+C_4 \mu(V) \sum_{n\in \N} b^{-(n-1)(p-\overline{Q})} \sum_{\gamma_i \in \mathcal{G}\cap \mathcal{A}(n)}  \nu(\Sh(\gamma_i x_0,2R)) .\]

{\em Using the definition of the sequence $\{\gamma_i\}$, we show that the corresponding shadows $\mathcal{S}(\gamma_ix_0,2R)$ are pairwise disjoint when $\gamma_i$ is restricted to any annulus $\mathcal{A}(n)$.
} Let then $\gamma_i\neq \gamma_j$ and assume there is some $\xi \in \mathcal{S}(\gamma_i x_0,2R)\cap \mathcal{S}(\gamma_j x_0,2R)$. Let $c_1,c_2$ be geodesic rays starting in $x_0$ in the class of $\xi$ such that $c_1$ intersects $B(\gamma_i x_0,2R)$ and $c_2$ intersects $B(\gamma_j x_0,2R)$. Let $P_1\in c_1$ be the point closest to $\gamma_i x_0$. It is clear that $d_X(P_1,\gamma_ix_0)\leq 2R$ and $n-1-2R \leq d_X(P_1,x_0) \leq n+2R$ because $\gamma_i\in \mathcal{A}(n)$. Similarly, choosing $P_2\in c_2$ the point closest to $\gamma_j x_0$, we see that $d_X(P_2,\gamma_j x_0)\leq 2R$ and $n-1-2R\leq d_X(P_2,x_0)\leq n+2R$. As $c_1$ and $c_2$ are both geodesic rays in the class of $\xi$ that start in $x_0$, they lie at Hausdorff-distance less than $\delta$ from each other. So, $d_X(P_1,P_2)\leq 4R+1+2\delta$ and so $d_X(\gamma_ix_0,\gamma_jx_0)\leq 8R+1+2\delta$. We thus obtain $d_K(\gamma_i,\gamma_j)\leq A(8R+1+2\delta+B)$ and so $d_V(\gamma_i,\gamma_j)\leq 1$, which leads to a contradiction as by construction $d_V(\gamma_i,\gamma_j)\geq 2$ for all $\gamma_i,\gamma_i \in \mathcal{G}$. We can thus bound $\sum_{\gamma \in \mathcal{G}\cap \mathcal{A}(n)}  \nu(\Sh(\gamma_i x_0,2R))\leq \nu(\partial G)$ and obtain
\[ \|c(k)\|_p^p \leq  D_1+C_4 \mu(V) \nu(\partial G) \sum_{n\in \N} b^{-(n-1)(p-\overline{Q})}.\]
Hence, $c(g)\in L_p(G,\mu)$ for all $g\in G$ and $p>\overline{Q}$.

The proof that $c$ is proper can then be copied word for word from the proof of Lemma \ref{lm:properness}. 
\end{proof}
Using that the Ahflors regular conformal dimension of the boundary of $Sp(n,1)$ and $F_4^{-20}$ are known to be $4n+2$ and $22$ respectively \cite{Pansu}, we obtain the following corollary.
\begin{corollary}
The group $G=F_4^{-20}$ for $p>22$ and the groups $G=Sp(n,1)$ for $p>4n+2$ admit proper affine isometric actions on a finite $l_p$-direct sum of copies of $L_p(G)$. Moreover, the linear part of the action is the direct sum of the natural actions by translation on the copies of $L_p(G)$ and the corresponding $1$-cocycle has compression at least $1/p$. 
\end{corollary}
\section{Compression of locally compact property $(A)$ groups}
The following definition is due to U. Haagerup and A. Przybyszewska \cite{Haagerup}.

\begin{defn}
 Let $G$ be a topological group. A \plig metric $d$ on $G$ is a metric on $G$, which is proper, left invariant and generates the topology.
\end{defn}

\begin{defn}
 Let $(X,d_X)$ and $(Y,d_Y)$ be metric spaces.
 \begin{itemize}
  \item A map $f \colon X \to Y$ is called \emph{uniformly expansive} if
  \[
   \forall R > 0 \, \exists S > 0 \quad \mbox{such that} \quad d_X(x,x')\leq R \Rightarrow d_Y(f(x),f(x')) \leq S.
  \]
  \item A map $f \colon X \to Y$ is called \emph{metrically proper} if
  \[
   \forall B \subset Y \quad \mbox{$B$ is bounded} \Rightarrow f^{-1}(B) \subset X \mbox{ is bounded.}
  \]
  \item A map $f \colon X \to Y$ is \emph{coarse} if it is metrically proper and uniformly expansive.
  \item Two coarse maps $h_{0}, h_{1} \colon X \to Y$ are \emph{coarsely equivalent} when
  \[
   \exists C > 0 \, \forall x \in X \quad d_Y(h_0(x), h_1(x)) < C.
  \]
  We denote the coarse equivalence relation by $\sim_{c}$.
  \item The spaces $X$ and $Y$ are \emph{coarsely equivalent} if there exist coarse maps $f \colon X \to Y$ and $g \colon Y \to X$ such that
  \[
   f \circ g \sim_c \mathrm{Id}_Y \quad \mathrm{and} \quad g \circ f \sim_c \mathrm{Id}_X
  \]
 \end{itemize}
\end{defn}

\begin{theorem}\cite[Theorem 4.5. and Theorem 2.8.]{Haagerup}
 Every locally compact, second countable group $G$ has a \plig metric. Moreover any two \plig metrics $l_1,l_2$ are coarsely equivalent for $f:(G,l_1)\rightarrow (G,l_2)$ and $g:(G,l_2)\rightarrow (G,l_1)$ in the above definition as the identity map.
\end{theorem}

\begin{defn}
 Let $(G,d)$ be a locally compact second countable group with a \plig metic, and let $\mu$ denote the Haar measure on $G$. Then we say that \emph{the $d$-balls have exponentially controlled growth} if there exists constants $\alpha, \beta >0$ such that
 \[
  \mu(B_d(e,n)) \leq \beta e^{\alpha n}
 \]
for all $n \in \N$.
\label{def:expcongro}
\end{defn}

\begin{theorem}\cite[Theorem 5.3.]{Haagerup}
 Every locally compact second countable group $G$ has a \plig metric on $G$, for which the $d$-balls have exponentially controlled growth.
\end{theorem}
Whenever we talk about metric notions, such as the compression or the coarse embeddability of a locally compact, second countable group $G$, then we always refer to $G$ being equipped with such a \plig metric whose $d$-balls have exponentially controlled growth. In this context, we can always assume that coarse embeddings are continuous: 
\begin{prop} \cite[Proposition 3.3.]{DL1}
 Let $G$ be a locally compact, second countable group. The following are 
equivalence:
 \begin{itemize}
  \item $G$ admits a coarse embedding into a Hilbert space;
  \item $G$ admits a Borel coarse embedding into a Hilbert space;
  \item $G$ admits a continuous coarse embedding into a Hilbert space.
 \end{itemize}
\end{prop}

Moreover, as far as compression is concerned, we can always restrict our attention to continuous coarse embeddings:
\begin{lemma}
 Let $G$ be a locally compact, second countable group and $d$ a \plig  metric. Let $f \colon G \to \Hs$ be a large scale Lipschitz map. Then there exists a continuous large-scale Lipschitz map $\widehat f$ such that
 \[
 R(f) = R(\widehat f).
 \]
\end{lemma}

\begin{proof}
 In the proof of \cite[Proposition 3.3.]{DL1}, the authors construct a continuous function $\widehat f \colon G \to \Hs$ such that there exists $R > 0$ with
 \[
  \norm{f(x) - \widehat f(x)}_{\Hs} \leq R
 \]
 for all $x \in G$. Hence $\widehat f$ is large-scale Lipschitz and has the same compression as $f$.
\end{proof}

We are interested in the connections between compression and property $(A)$.
\begin{defn}[Property A]
 Let $G$ be a locally compact second countable group. Then $G$ has \emph{property A} if for any compact subset $1\in K \subset G$ and $\delta > 0$ there exists a compact subset $L \subset G$ and a (continuous) positive type kernel $k \colon G \times G \to \C$ such that $\mathrm{supp}(k) \subset \mathrm{Tube}(L)$ and
 \[
  \sup |k(s,t) - 1| < \delta,
 \]
 for every $s,t\in G$ with $s^{-1}t\in K$.
\end{defn}
We wrote the word ``continuous'' in brackets as the definition is equivalent to when we omit continuity.
%
%
%

 %

Let $G$ be a locally compact, second countable group. Given a measurable complex-valued kernel $k \colon G \times G \to \C$, define an operator $\mathrm{Op}(k) \colon L^2(G) \to L^2(G)$ by convolution
\[
 \mathrm{Op}(k) \xi(x) = \int_G k(x,y) \xi(y) \, d\mu(y)
\]

\begin{prop}
 Under either of the following conditions, $\mathrm{Op}(k)$ is a bounded operator.
 \begin{itemize}
  \item[(i)] $k$ is bounded and has compact width.
  \item[(i)] Let $k$ be non-negative real-valued with the property that there exists $C > 0$ such that
  \begin{align*}
   \int_G k(s,t) \, d\mu(s) & \leq C, \quad \mbox{for all $t \in G$}\\
   \int_G k(s,t) \, d\mu(t) & \leq C, \quad \mbox{for all $s \in G$}.
  \end{align*}
Then $\mathrm{Op}(k)$ is bounded and $\norm{\mathrm{Op}(k)} \leq C$.
 \end{itemize}
\end{prop}

\begin{proof}
 We shall only prove (i) as (ii) is a special case of the Schur Test. We aim to show that there exists a constant $M>0$ such that for each $f \in L^2(G):$ 
 \[
  \norm{\mathrm{Op}(k)f} \leq M \norm{f}_{L^2(G)}.
 \]
Take a compact subset $L\subset G$ such that $\supp{k} \subset \mathrm{Tube}(L)$. Further take $K > 0$ such that $k(x,y) \leq K$ for all $x,y \in G$. Noting that $k(x,y) \neq 0$ implies $y \in xL$, we can check that:
 \begin{align*}
  \norm{\mathrm{Op}(k)f}^2 & = \int_G \left( \int_G k(x,y) f(y) \, d\mu(y) \right)^2 d \mu(x)\\
  & = \int_G \left( \int_{xL} k(x,y) f(y) \, d\mu(y) \right)^2 d \mu(x) \\
  & \leq \int_G \left( \int_{xL} k(x,y)^2 \, d\mu(y) \int_{xL} f(y)^2 \, d\mu(y) \right) d \mu(x)\\
  & \leq \mu(L) K^2 \int_G \int_{xL} f(y)^2 \, d\mu(y) d \mu(x)\\
  & = \mu(L) K^2 \int_L \int_G f(xy)^2 \, d \mu(x) d \mu(y) \\
  & \leq \mu(L) K^2 \int_L \norm{f}^2 \, d \mu(y) \\
  & = \mu(L)^2 K^2 \norm{f}^2 \qedhere
 \end{align*}
\end{proof}

We are now in the position to state the main result of this section.
\begin{theorem}
 Let $G$ be a locally compact second countable group and let $d$ be a \plig metric on $G$ with exponentially controlled growth of balls. If $\alpha(G)>1/2$ then $G$ has property A.  \label{theorem:PropertyA}
\end{theorem}

\begin{proof}
 For any sufficiently small $\varepsilon>0$, there exists a continuous large-scale Lipschitz function $f \colon G \to \Hs$ such that
 \[
  d(x,y)^{\frac{1+\varepsilon}{2}} \leq \norm{f(x) - f(y)}_{\Hs} , 
 \]
 when $d(x,y)$ is sufficiently large. For $\kappa \geq 1$, define a positive definite function $u_{\kappa} \colon G \times G \to \R$ by
 \[
  u_{\kappa}(s,t) = e^{-\norm{f(x) - f(y)}^2\kappa^{-1}}
 \]
 Let $\mathcal{A}$ be the $C^*$-algebra of bounded operators on $L^2(G)$ generated by the operators $\mathrm{Op}(k)$, where $k$ runs over all bounded compact width kernels.

 \begin{lemma}
  The operators $\mathrm{Op}(u_{\kappa})$ are in $\mathcal{A}$ for all $\kappa \geq 1$.
 \end{lemma}
 \begin{proof}
 We show that for every $\kappa \geq 1$ the kernel $u \colon G \times G \to \C$ defined by
 \[
  u(s,t) = e^{-\norm{f(s) - f(t)}^2 \kappa^{-1}}
 \]
defines an element in $\mathcal{A}$. Define for $n \in \N$,
\[
k_n(s,t) = \begin{cases} u(s,t), & \quad \mbox{if $d(s,t) > n$}\\
 0 & \quad \mbox{otherwise}
\end{cases}
\]
It follows that $u - k_n$ is a bounded compact width kernel so $\mathrm{Op}(u-k_n) \in \mathcal{A}$. Since $\mathrm{Op}(u) - \mathrm{Op}(u-k_n) = \mathrm{Op}(k_n)$ on compact supported elements of $L^2(G)$, it suffices to show that $\norm{\mathrm{Op}(k_n)} \to 0$ as $n \to \infty$. Using the Schur test, it suffices to show that
there exists $C > 0$ such that $\int_G u(s,t) \, d\mu(t) \leq C$ for all $s \in G$. 
As $d$ is a \plig metric, there exists $\alpha, \beta > 0$ such that $\mu(B(1,n)) \leq \beta e^{\alpha n}$ for all $n > 0$. Hence $\mu(B(1,n+1) \setminus B(1,n)) \leq \beta e^{\alpha (n+1)}$. Choose an $m \in \N$ large enough so that $e^{\alpha} < e^{{\kappa^{-1} m}^{\varepsilon}}$ and if $d(s,t) > m$ then
\[
 d(s,t)^{(1+\varepsilon)/2} \leq \norm{f(s) - f(t)}.
\]
Then, for some constant $C'$, we have
\begin{align*}
 \int_G u(s,t) \, d\mu(t) & = \int_{B(s,m)} u(s,t) \, d\mu(t) + \int_{B(s,m)^c} u(s,t) \, d\mu(t) \\
 & \leq C' \mu(B(s,m)) + \sum_{n\geq m} \int_{B(s,n+1) \setminus B(s,n)} u(s,t) \, d\mu(t) \\
 & \leq C' \mu(B(1,m)) + \sum_{n\geq m} \int_{B(s,n+1) \setminus B(s,n)} e^{-\kappa^{-1} n^{1 + \varepsilon} } \, d \mu(t)\\
 & \leq C' \mu(B(1,m)) + \sum_{n\geq m} \mu(B(s,n+1) \setminus B(s,n)) e^{-\kappa^{-1} n^{1 + \varepsilon} } \\
 & \leq C' \mu(B(1,m)) + \beta e^{\alpha} \sum_{n\geq m} \left( \frac{e ^{\alpha}}{e^{\kappa^{-1} n^{\varepsilon}}} \right)^n \\
 & \leq C' \mu(B(1,m)) + \beta e^{\alpha} \sum_{n\geq m} \left( \frac{e ^{\alpha}}{e^{\kappa^{-1} m^{\varepsilon}}} \right)^n :=C <\infty.
\end{align*}
\end{proof}
 
 \begin{lemma}
  The operators $\mathrm{Op}(u_k) \colon L^2(G) \to L^2(G)$ are positive for all $k \geq 1$
 \end{lemma}
 
 \begin{proof}
 We show that for every $\kappa > 0$ the kernel $u \colon G \times G \to \C$ defined by
 \[
  u(s,t) = e^{-\norm{f(s) - f(t)}^2 \kappa^{-1}}
 \]
 is a positive element in $\mathcal{A}$. By \cite[Theorem C.1.4.]{BHV08}, as $u$ is positive there exists a Hilbert space and a continuous map $h \colon G \to \Hs$ such that
 \[
  u(x,y) = \inner{h(x)}{h(y)}_{\Hs},
 \]
 for all $x,y \in G$.
 We aim to show that $\inner{\mathrm{Op}(u) \xi}{\xi} \geq 0$ for all $ \xi \in L^2(G)$. It suffices to show that this inequality holds for all compactly supported functions on $G$ as these form a dense subset of $L^2(G)$. Fix $\xi \in L^2(G)$ which is supported on a compact subset $K\subset G$. As $h$ is continuous there exists $M > 0$ such that $\norm{h(x)}_{\Hs} \leq M^{1/2}$ for all $x \in K$. For any $\eta \in \Hs$, define a function $\varphi \colon G \to \C$ by
 \[
  \varphi(y) = \inner{\eta}{h(y) \xi(y)}_{\Hs}.
 \]
 This is a square integrable function. Indeed
 \begin{align*}
  \int_G |\inner{\eta}{h(y) \xi(y)}_{\Hs}|^2 \, d\mu(y) & = \int_K |\inner{\eta}{h(y) \xi(y)}_{\Hs}|^2 \, d\mu(y) \\
	& \leq \int_K \norm{\eta}_{\Hs}^2 \norm{h(y) \xi(y)}_{\Hs}^2 \, d \mu(y) \\
  & \leq \norm{\eta}_{\Hs}^2 \int_K |\xi(y)|^2 \norm{h(y)}_{\Hs}^2 \, d\mu(y)\\
  & \leq \norm{\eta}_{\Hs}^2 \norm{\xi}^{2}_{L^2} M \mu(K)
 \end{align*}
By compactness of $K$, one has $L^2(K)\subset L^1(K)$, so $\varphi \in L^1(K)$.  This implies that $\varphi^{1/2}$ is also square-integrable. We can thus define a functional $\Hs \to \C$ by
 \[
  \eta \mapsto \int_K \inner{\eta}{h(y) \xi(y)}_{\Hs} \, d \mu(y)
 \]
This functional is clearly linear and it is continuous as
\begin{eqnarray*}
 |\int_K \inner{\eta_1-\eta_2}{h(y) \xi(y)}_{\Hs} \, d\mu(y) |^2 & \leq &
\int_K |\inner{\eta_1-\eta_2}{h(y) \xi(y)}_{\Hs}|^2 \, d\mu(y) \\
&  \leq & \norm{\eta_1-\eta_2}_{\Hs}^2 \norm{\xi}^{2}_{L^2} M \mu(K).
\end{eqnarray*}
By the Riesz representation theorem, there exists a unique vector, denoted here by $\int_K h(y) \xi(y) \, d\mu(y) \in \Hs$, such that
\[
 \inner{\eta}{\int_K h(y) \xi(y) \, d\mu(y)} = \int_K \inner{\eta}{h(y) \xi(y)}_{\Hs} \, d\mu(y).
\]
Using this idea twice, we get that 
\begin{align*}
 \inner{\mathrm{Op}(u) \xi}{\xi}_{L^2}  & = \int_K \int_K u(x,y) \xi(y) \overline{\xi(x)} \, d\mu(y) d\mu(x) \\
 & =  \int_K \int_K \inner{h(x) \xi(x)}{h(y) \xi(y)}_{\Hs} \, d \mu(y) d\mu(x) \\
 & = \inner{\int_K h(x) \xi(x) \, d\mu(x)}{\int_K h(x) \xi(x) \, d\mu(x)}_{\Hs} \geq 0
\end{align*}
 \end{proof}
We have that $\mathrm{Op}(u_\kappa)$ are positive operators for all $\kappa \geq 1$. Let $V_\kappa \in \mathcal{A}$ be the positive square root of $\mathrm{Op}(u_\kappa)$ and let $W_\kappa \in \mathcal{A}$ be operators represented by compact width kernels and such that $\norm{V_\kappa - W_\kappa} \norm{V_\kappa} \to 0$. Define kernels $\widehat u_\kappa $ by
\[
 \widehat u_\kappa (s,t) = \frac{1}{\mu(B(e,1/\kappa))^2} \inner{W_\kappa (\chi_{B(s,1/\kappa)})}{W_\kappa (\chi_{B(t,1/\kappa)})}_{L^2(G)}
\]
One verifies easily that this is a positive kernel. Let us show it is of compact width.
\begin{lemma} The positive kernels $\widehat u_\kappa$ have compact width. \end{lemma}
\begin{proof}
To lighten notation, let us fix a $\kappa>0$ and write $W:= W_\kappa=\mathrm{Op}(w)$ where $w$ is a bounded kernel with compact with. Denote $L\subset G$ a compact set such that $\supp{w} \subset \mathrm{Tube}(L)$ and choose an element $a \in L^{-1}$. Then for any $s \in G$
\[
 \sup \set{d(sa,x) : x \in s B(e,1/\kappa) L^{-1}} \leq l(a) + 1/\kappa + M \leq 2M + 1/\kappa
\]
Hence $s B(e,1/\kappa) L^{-1} \subset B(sa, 2M + 1/\kappa)$ for all $s \in G$. Thus if $d(sa, ta) > 4M + 2/\kappa$ for some $t\in G$, then 
\[
s B(e,1/\kappa) L^{-1} \cap  t B(e,1/\kappa) L^{-1} = \emptyset
\]
As $d(s,t) \leq d(s,sa) + d(sa,ta) + d(ta,t)$ it follows that $d(sa,ta) \geq d(s,t) - 2M$, so if $d(s,t) > 6M +2/\kappa$ then  $sB(e,1/\kappa) L^{-1} \cap  tB(e,1/\kappa) L^{-1} = \emptyset$. We have that
\begin{align*}
 \mu(B(e,1/\kappa))^2  \widehat u_\kappa (s,t) & = \inner{W_\kappa \chi_{B(s,1/\kappa)}}{W_\kappa \chi_{B(t,1/\kappa)}}_{L^2(G)}\\
  & = \int_G \mathrm{Op}(w) \chi_{B(s,1/\kappa)}(x) \cdot \overline{\mathrm{Op}(w) \chi_{B(t,1/\kappa)}(x)} \, d\mu(x)\\
  & = \int_G \left[ \int_{B(s,1/\kappa)} w(x,y) \, d\mu(y) \right] \cdot \overline{\left[ \int_{B(t,1/\kappa)} w(x,z) \, d\mu(z) \right]} \, d\mu(x)
\end{align*}
As $\supp{w} \subset \mathrm{Tube}(L)$, we deduce that $\supp{\widehat u_\kappa} \subset \mbox{Tube}(B(1, 6M + 2/\kappa))$.
\end{proof}
To finish the proof, we now prove that one can bound $\lvert u_k(s,t) - \widehat u_k(s,t)\rvert$ uniformly over any Tube$(K)$ with $K$ compact. Note that for any kernel $v_k$, the triangle inequality shows that 
\[ \lvert u_k(s,t) - \widehat u_k(s,t)\rvert \leq \lvert u_k(s,t)-v_k(s,t) \rvert + \lvert v_k(s,t)-\widehat u_k(s,t)  \rvert. \]
Let us define $v_k$ by
\[
 v_k(s,t) = \frac{1}{\mu(B(e,1/k))^2} \inner{\mathrm{Op}(u_k) \chi_{B(s,1/k)}}{ \chi_{B(t,1/k)}}_{L^2(G)}.
\]

Then,
\begin{align*}
 |v_k(s,t) - \widehat u_k(s,t)| & = \frac{1}{\mu(B(e,1/k))^2}|\inner{(\mathrm{Op}(u_k) - W_k^* W_k) \chi_{B(s,1/k)}}{\chi_{B(t,1/k)}}| \\
 & \leq \norm{\mathrm{Op}(u_k) - W_k^* W_k} = \norm{V_k^* V_k - W_k^* W_k}\\
 & \leq \norm{V_k - W_k}(\norm{V_k} + \norm{W_k})\\
 & \leq \norm{V_k - W_k}(1\norm{V_k} + \norm{V_k - W_k}),
\end{align*}
which tends to zero as $k \to \infty$.
It now suffices to show that for any $K \subset G$ compact and $\varepsilon > 0$ there exists a $k_0 \in \N$ such that for all $k \geq k_0$
\[
\sup_{(s,t) \in \mathrm{Tube}(K)} |u_k(s,t) - v_k(s,t)| < \varepsilon. 
\]
As $f$ is a coarse embedding, there exists a constant $C=C_K$ such that for all $(s,t) \in \mathrm{Tube}(K), k\in \R, x \in B(s,1/k)$ and $y \in B(t,1/k)$
\[
 |\norm{f(s) - f(t)}^2 k^{-1} - \norm{f(x) - f(y)}^2 k^{-1}| < C / k
\]
We choose $k$ large enough so that $C/k < \varepsilon$. For any $a,b \geq 0$ we have that $|e^{-a} - e^{-b}| \leq |a - b|$, so $|e^{\norm{f(s) - f(t)}^2 k^{-1}} -  e^{\norm{f(x) - f(y)}^2 k^{-1}}| < \varepsilon$ and so
\begin{align*}
 |u_k(s,t) - v_k(s,t)| & =  |u_k(s,t) - \frac{1}{\mu(B(e,1/k))^2} \int_{B(s,1/k)} \int_{B(t,1/k)} u_k(x,y) \, d\mu(x) d\mu(y) | \\ 
 & \leq  \frac{1}{\mu(B(e,1/k))^2} \int_{B(s,1/k)} \int_{B(t,1/k)}| u_k(s,t) -  u_k(x,y)| \, d\mu(x) d\mu(y) \\
 & = \frac{1}{\mu(B(e,1/k))^2} \cdot \\ 
&  \int_{B(s,1/k)} \int_{B(t,1/k)}| e^{-\norm{f(s) - f(t)}^2k^{-1}} -  e^{-\norm{f(x) - f(y)}^2k^{-1}}| \, d\mu(x) d\mu(y)\\
 & \leq \varepsilon
\end{align*}
 As
\begin{multline*}
\sup_{(s,t) \in \mathrm{Tube}(K)}|u_k(s,t) - \widehat u_k(s,t)|  \leq \sup_{(s,t) \in \mathrm{Tube}(K)} |u_k(s,t) -v_k(s,t)|\\ + \sup_{(s,t) \in \mathrm{Tube}(K)} |v_k(s,t) - \widehat u_k(s,t)|,
\end{multline*}
we conclude that $\widehat u_k$ uniformly converges to 1 on sets of the form Tube$(K)$ with $K$ compact. As the $\widehat u_k$ are all finite width kernels, $G$ has property $(A)$.
\end{proof}

A direct generalization of Theorem $4.1$ in \cite{CTV} shows that the equivariant counterpart of our main result also generalizes to the locally compact setting. This generalization is straightforward so we omit the proof here.
\begin{prop}
 Let $G$ be a locally compact second countable group and let $d$ be a \plig metric with exponentially controlled growth of balls. If $G$ admits a proper affine isometric action on a Hilbert space such that the $1$-cocycle $b$ has compression $R(b,d) > 1/2$, then $G$ is amenable.
\end{prop}

\section*{Acknowledgements}
{We would like to thank Marc Bourdon, Jacek Brodzki, Yves de Cornulier, Steven Deprez, Bogdan Nica, Kang Li and Alain Valette for useful comments and remarks on this manuscript. We also thank Ian Leary and Romain Tessera for an interesting discussion and Pierre-Emmanuel Caprace for introducing the second author to locally compact hyperbolic groups.}

\end{document}